\documentclass[11pt]{amsart}

\usepackage[draft]{changes}

\definechangesauthor[color=red]{pe}
\definechangesauthor[color=red]{xm}
\definechangesauthor[color=red]{hx}
\newcommand{\stkout}[1]{\ifmmode\text{\sout{\ensuremath{#1}}}\else\sout{#1}\fi}
\setdeletedmarkup{\stkout{#1}}
\colorlet{Changes@Color}{red}

\usepackage{amssymb}

\usepackage{graphics}
\usepackage{graphicx}
\usepackage{amsmath}
\usepackage{amsthm}
\usepackage{amsfonts}
\usepackage{mathtools}
\usepackage{xcolor}
\usepackage[english]{babel}
\usepackage[margin=1.0in]{geometry}
\usepackage[colorlinks=true,
        linkcolor=blue]{hyperref}
\usepackage{bbm}
\usepackage{verbatim}
\usepackage{extarrows}
\usepackage{blkarray}
\usepackage{tensor}

\usepackage[T1]{fontenc}

\numberwithin{equation}{section}

\newtheorem{prop}{Proposition}
\newtheorem{lemma}[prop]{Lemma}

\newtheorem{thm}[prop]{Theorem}
\newtheorem*{thm*}{Theorem}

\numberwithin{prop}{section}

\newtheorem{defn}[prop]{Definition}
\newtheorem*{defn*}{Definition}
\theoremstyle{definition}

\newtheorem*{ex*}{Example}
\newtheorem{rmk}[prop]{Remark}

\definecolor{c1}{rgb}{0.2,0.4,0.5}
\definecolor{c2}{rgb}{0.1,0.3,0.5}
\definecolor{c3}{rgb}{0.2,0.7,0.5}
\usepackage{tikz}

\def \k {K\"ahler }
\def \ke {K\"ahler--Einstein }
\newcommand{\oo}[1]{\overline{#1}}

\newcommand{\nab}{\nabla}

\newcommand{\bC}{\mathbb{C}}

\newcommand{\dbar}{\oo\partial}

\DeclareFontFamily{U}{MnSymbolC}{}
\DeclareSymbolFont{MnSyC}{U}{MnSymbolC}{m}{n}
\DeclareFontShape{U}{MnSymbolC}{m}{n}{
	<-6>  MnSymbolC5
	<6-7>  MnSymbolC6
	<7-8>  MnSymbolC7
	<8-9>  MnSymbolC8
	<9-10> MnSymbolC9
	<10-12> MnSymbolC10
	<12->   MnSymbolC12}{}
\DeclareMathSymbol{\intprod}{\mathbin}{MnSyC}{'270}

\DeclareMathOperator{\tr}{tr}

\DeclareMathOperator{\wt}{wt}

\begin{document}

\title[]{On the analytic and geometric aspects of obstruction flatness}

\begin{abstract} 
In this paper, we investigate analytic and geometric properties of obstruction flatness of strongly pseudoconvex CR hypersurfaces of dimension $2n-1$. Our first two results concern local aspects. Theorem \ref{thm:obsosc} asserts that any strongly pseudoconvex CR hypersurface $M\subset \bC^n$ can be osculated at a given point $p\in M$ by an obstruction flat one up to order $2n+4$ generally and $2n+5$ if and only if $p$ is an obstruction flat point. In Theorem \ref{existence of obstruction flat hypersurface with transverse symmetry Prop}, we show that locally there are non-spherical but obstruction flat CR hypersurfaces with {\em transverse symmetry} for $n=2$. The final main result in this paper concerns the existence of obstruction flat points on compact, strongly pseudoconvex, 3-dimensional CR hypersurfaces.  Theorem \ref{existence of obstruction flat points on circle bundle Prop} asserts that the unit sphere in a negative line bundle over a Riemann surface $X$ always has at least one circle of obstruction flat points.
\end{abstract}



\author [Ebenfelt]{Peter Ebenfelt}
\address{Department of Mathematics, University of California at San Diego, La Jolla, CA 92093, USA}
\email{{pebenfelt@ucsd.edu}}

\author[Xiao]{Ming Xiao}
\address{Department of Mathematics, University of California at San Diego, La Jolla, CA 92093, USA}
\email{{m3xiao@ucsd.edu}}

\author [Xu]{Hang Xu}
\address{School of Mathematics (Zhuhai), Sun Yat-sen University, Zhuhai, Guangdong 519082, China}
\email{{xuhang9@mail.sysu.edu.cn}}

\thanks{The first author was supported in part by the NSF grant DMS-1900955 and DMS-2154368. The second author was supported in part by the NSF grants DMS-1800549 and DMS-2045104. The third author was supported in part by the NSFC grant No. 12201040.}

\maketitle
	

\section{Introduction}


		
	

Let $\Omega\subset \mathbb{C}^n$ be a smoothly bounded, strongly pseudoconvex domain and consider the Dirichlet problem
\begin{equation}\label{Dirichlet problem}
\begin{dcases}
J(u):=(-1)^n \det \begin{pmatrix}
u & u_{\bar z_k}\\
u_{z_j} & u_{z_j\bar z_k} \\
\end{pmatrix}=1 & \mbox{in } \Omega\\
u=0 & \mbox{on } \partial\Omega \\
u>0 & \mbox{in } \Omega.
\end{dcases}
\end{equation}
If $u$ is a solution of \eqref{Dirichlet problem}, then $-\log u$ is the \k potential of a complete \ke metric in $\Omega$ with negative Ricci constant. The existence and uniqueness of a solution $u\in C^{\infty}(\Omega)$ to \eqref{Dirichlet problem} was established by S.-Y. Cheng and S.-T. Yau \cite{CheYau}. Prior to that, C. Fefferman \cite{Fe2} had constructed approximate solutions $\rho\in C^{\infty}(\oo{\Omega})$ to \eqref{Dirichlet problem} that satisfy $J(\rho)=1+O(\rho^{n+1})$, and shown that such $\rho$ are unique modulo $O(\rho^{n+2})$. Fefferman's approximate solutions $\rho$ are called {\em Fefferman defining functions} and the exact solution $u$ is referred to as the {\em Cheng--Yau solution}. Subsequently,
J. Lee and R. Melrose \cite{LeeMel82} showed that the Cheng--Yau solution $u$ admits an asymptotic expansion of the form
\begin{equation}\label{Lee--Melrose expansion}
u\sim \rho\sum_{k=0}^{\infty} \eta_k\bigl(\rho^{n+1}\log\rho\bigr)^k,
\end{equation}
where each $\eta_k$ is in $C^{\infty}(\oo{\Omega})$ and $\rho$ is a Fefferman defining function.
It follows from \eqref{Lee--Melrose expansion} that, in general, the Cheng--Yau solution $u$ can only possess a finite degree of boundary smoothness; namely, $u$ is in $C^{n+2-\varepsilon}(\oo{\Omega})$ for any $\varepsilon>0$. C. R. Graham discovered that in the expansion (\ref{Lee--Melrose expansion}), the restriction of $\eta_1$ to the boundary, $\eta_1|_{\partial\Omega}$, turns out to be precisely the obstruction to $C^{\infty}$ boundary regularity of the Cheng--Yau solution. To be more precise, in \cite{Graham1987a} Graham proved that if $\eta_1|_{\partial\Omega}$ vanishes identically on an open subset $U\subset \partial\Omega$, then $\eta_k$ vanishes to infinite order on $U$ for every $k\geq 1$. For this reason, $\eta_1|_{\partial\Omega}$ is called the \emph{obstruction function}. 

Graham also showed ({\em op.\ cit.}) that, for any $k\geq 1$, the coefficients $\eta_k$ mod $O(\rho^{n+1})$ are independent of the choice of Fefferman defining function $\rho$ and are determined near a point $p\in\partial\Omega$ by the local CR geometry of $\partial\Omega$ near $p$. As a consequence, the $\eta_k$ mod $O(\rho^{n+1})$, for $k\geq 1$, are local CR invariants that can be defined on any strongly pseudoconvex CR hypersurface in a complex manifold. In particular, the obstruction function $\mathcal{O}:=\eta_1|_{\partial\Omega}$ is a local CR invariant that can be defined on any such strongly pseudoconvex CR hypersurface $M$ (see Graham \cite{Graham1987a}). If the obstruction function $\mathcal{O}$ vanishes at $p \in M$, then $p$ is called an \emph{obstruction flat point}. If $\mathcal{O}$
vanishes identically on $M$, then $M$ is said to be \emph{obstruction flat}. We note that the definition of obstruction flat point only requires the obstruction function $\mathcal{O}$ to vanish rather than being flat, but the terminology is motivated by Theorem \ref{thm:obsosc} below. 

The most basic examples of obstruction flat hypersurfaces are the unit sphere in $\bC^n$ and, more generally, any spherical CR hypersurface; recall that a CR hypersurface $M$ is called {\em spherical} if, at every $p\in \Sigma$, there is a neighborhood of $p$ that is CR diffeomorphic to an open piece of the sphere. Graham \cite{Graham1987a, Graham1987b} showed that locally there are many non-spherical but obstruction flat hypersurfaces. There are, however, no known examples of smoothly bounded, strongly pseudoconvex domains $\Omega\subset \mathbb{C}^n$ whose boundaries are non-spherical but obstruction flat. In the more general context of compact strongly pseudoconvex CR manifolds (always of hypersurface type in this paper) $M=M^{2n-1}$ of dimension $2n-1$, there are non-spherical but obstruction flat examples for $n\geq 3$ (i.e., $\dim M\geq 5$), as was recently demonstrated by the authors \cite{EXX22}, but there are no known examples of such for $n=2$ ($\dim M=3$). Indeed, the examples constructed in \cite{EXX22} for $n\geq 3$ are unit spheres in negative line bundles over compact K\"ahler manifolds of dimension $n-1$ and, for $n=2$, non-spherical but obstruction flat examples cannot exist in this context by a result of the first-named author \cite{Ebenfelt18}. 

If a compact, strongly pseudoconvex, 3-dimensional CR manifold $M$ bounds a domain in a complex manifold, the obstruction function $\mathcal O$ coincides (modulo a universal non-zero constant) with the boundary trace of the log term $\psi_B$ in Fefferman's asymptotic expansion of the Bergman kernel. It has been conjectured that every such compact, strongly pseudoconvex, 3-dimensional CR manifold $M$ which is obstruction flat is also spherical. This is often referred to as the {\em Strong Ramadanov Conjecture}, so named in accordance with the classical Ramadanov Conjecture, which asserts that $M$ is spherical if $\psi_B$ vanishes to {\em infinite order} along $M$. The Ramadanov Conjecture for boundaries $M$ of domains in 2-dimensional complex manifolds has been established in \cite{Graham1987b, Bou}, where it suffices to assume that $\psi_B$ vanishes to {\em second order} along $M$. Progress and further discussion of the Strong Ramadanov Conjecture can be found in \cite{Ebenfelt18,CE-AM2019, CE-AJM2021, CE-Crelle2021}.

In this paper, we investigate analytic and geometric properties of obstruction flatness of strongly pseudoconvex CR hypersurfaces of dimension $2n-1$. Our first two results concern local aspects. Theorem \ref{thm:obsosc} asserts that any strongly pseudoconvex CR hypersurface $M\subset \bC^n$ can be osculated at a given point $p\in M$ by an obstruction flat one up to order $2n+4$ generally and $2n+5$ if and only if $p$ is an obstruction flat point. We recall that osculation by spherical CR hypersurfaces is possible up to order $4$ ($6$ for $n=2$) generally and $5$ ($7$ for $n=2$) if and only if $p$ is CR umbilical. In Theorem \ref{existence of obstruction flat hypersurface with transverse symmetry Prop}, we show that locally there are non-spherical but obstruction flat CR hypersurfaces with {\em transverse symmetry} for $n=2$. For $n\geq 3$, there are even compact ones, as demonstrated by the examples constructed in \cite{EXX22}, but for $n=2$, any compact obstruction flat $M$ with a transverse symmetry is necessarily spherical \cite{Ebenfelt18}. Graham's local construction of non-spherical but obstruction flat CR hypersurfaces $M$ does not easily lend itself to ensuring that $M$ has a transverse symmetry. Thus, Theorem \ref{existence of obstruction flat hypersurface with transverse symmetry Prop} clarifies the necessity of compactness for the result in \cite{Ebenfelt18}. The final main result in this paper concerns the existence of obstruction flat points on compact, strongly pseudoconvex, 3-dimensional CR hypersurfaces. The analogous question of existence of umbilical points was raised by J. K. Moser and S.-S. Chern \cite{ChMo74}. Theorem \ref{existence of obstruction flat points on circle bundle Prop} asserts that the unit sphere in a negative line bundle over a Riemann surface $X$ always has at least one circle of obstruction flat points. The corresponding result for umbilical points in this context was established in \cite{ES17}, provided that $X$ is not a torus.

The paper is organized as follows. In section \ref{sec:2}, we review the Chern--Moser normal forms for strongly pseudoconvex CR hypersurfaces. In Section \ref{sec:3}, we discuss the osculation result in Theorem \ref{thm:obsosc}. Section \ref{sec:4} is dedicated to Theorem \ref{existence of obstruction flat hypersurface with transverse symmetry Prop} and Section \ref{sec:5} to Theorem \ref{existence of obstruction flat points on circle bundle Prop}. 


We conclude this introduction by giving two alternative characterizations of obstruction flat points.
The first is given by the following proposition, which is a simple consequence of the Lee-Melrose expansion \eqref{Lee--Melrose expansion} above.



\begin{prop}
The Cheng--Yau solution $u$ of a smoothly bounded, strongly pseudoconvex domain $\Omega$  extends $C^{n+2}$ to $p \in \partial \Omega$ along any smooth curve intersecting $\partial \Omega$ transversally at $p$  if and only if $\partial\Omega$ is obstruction flat at $p$.
\end{prop}
For the second, we recall that a (smooth) defining function $\rho$ of a strongly pseudoconvex hypersurface $M\subset \mathbb{C}^n$ (which we assume to be positive on the pseudoconvex side of $M$) is a Fefferman defining function if  $J(\rho)=1+O(\rho^{n+1})$. In \cite{Fe2}, an algorithm for constructing Fefferman defining functions was developed. Suppose $\psi$ is any defining function of $M$. Recursively define
\begin{align}\label{Fefferman algorithm}
	\begin{split}
		\psi_1&=J(\psi)^{-1/(n+1)} \psi, \\
		\psi_{p}&=\psi_{p-1} \Bigl(1+\frac{1-J(\psi_{p-1})}{p(n+2-p)} \Bigr) \quad \mbox{ for } 2\leq p\leq n+1.
	\end{split}
\end{align}
Then each $\psi_p$ satisfies $J(\psi_p)=1+O(\psi^p)$. In particular, $\psi_{n+1}$ is a Fefferman defining function. Moreover, by \cite{Graham1987a} (see equation (4.11) there) one has
\begin{equation}\label{Fefferman defining function and obstruction function}
	J(\rho)=1+(n+2)\widehat{\mathcal{O}} \rho^{n+1}+O(\rho^{n+2}),
\end{equation}
where $\widehat{\mathcal{O}}$ is any smooth extension of the obstruction function  $\mathcal{O}$ to some neighborhood of $M$.
Thus, we have the second characterization: 
\begin{prop}
	Given $p \in M$, a Fefferman defining function satisfies $J(\rho)=1+O(|z-p|^{n+2})$ along any smooth curve intersecting $M$ transversally at $p$ if and only if $M$ is obstruction flat at $p$.
\end{prop}

\section{Chern--Moser normal forms for strongly pseudoconvex hypersurfaces}\label{sec:2}

In this section, we recall the Chern--Moser normal form for strongly pseudoconvex hypersurfaces. Let $M\subset\mathbb{C}^n$ be a strongly pseudoconvex hypersurface and assume for a moment that $M$ is real analytic. In \cite{ChMo74}, it was proved that given a point $p\in M$, there exists a coordinate chart $(z, w)=(z_1,\cdots, z_{n-1}, w)\in \mathbb{C}^n$, vanishing at $p$, in which $M$ is defined by a convergent power series of the form
\begin{equation}\label{CM normal form}
	2u=|z|^2+\sum_{\underset{l\geq 0}{|\alpha|, |\beta|\geq 2}} A_{\alpha\bar{\beta}}^l \, z^{\alpha}\oo{z^{\beta}} v^l,
\end{equation}
where $w=u+iv$ and $\alpha, \beta$ are multi-indices in $\mathbb{Z}_{\geq 0}^{n-1}$. The coefficients $A_{\alpha\bar{\beta}}^l\in \mathbb{C}$ satisfy $\oo{A_{\alpha\bar{\beta}}^l}=A_{\beta\bar{\alpha}}^l$ and certain trace conditions.
If we introduce
\begin{equation*}
A_{p\bar{q}}^l(z,\bar{z}):=\sum_{|\alpha|=p, |\beta|=q} A_{\alpha\bar{\beta}}^l\, z^{\alpha}\oo{z^{\beta}},
\end{equation*}
then the trace conditions can be formulated as follows: For each $l$,
\begin{equation*}
\Delta\bigl( A_{2\bar{2}}^l(z,\bar{z}) \bigr)=\Delta^2 \bigl( A_{2\bar{3}}^l(z,\bar{z}) \bigr)=\Delta^3\bigl( A_{3\bar{3}}^l(z,\bar{z}) \bigr)=0,
\end{equation*}
where $\Delta$ is the standard Laplace operator in $z$.
Equation \eqref{CM normal form} is called a {\em Chern-Moser normal form} of $M$ at $p$. Chern-Moser normal forms are unique modulo an action of the finite dimensional Lie group of automorphisms of the (spherical) hyperquadric $2u=|z|^2$ that fix the origin. In a Chern--Moser coordinate chart, one assigns weights to the coordinates as follows: $\wt(z_j)=1$, for $j=1,\ldots, n-1$, and $\wt(w)=2$. For a product, the weight is the sum of weights of all factors. When $M$ is merely smooth, one can find, for any $m\geq 2$, holomorphic charts such that $M$ is given in Chern--Moser normal form modulo terms of weight $m+1$. Alternatively, one can find formal charts such that $M$ is given by a formal power series of the form \eqref{CM normal form}. 

When $n=2$, the above conditions yield $A_{2\bar{2}}^l=A_{2\bar{3}}^l=A_{3\bar{2}}^l=A_{3\bar{3}}^l=0$ for any $l\geq 0$, and (the symmetric tensor associated to) the first nontrivial term $A_{2\bar{4}}^0$ is called E. Cartan's 6th order tensor. In the case $n \geq 3$, (the symmetric tensor associated to)  $A_{2\bar{2}}^0$ is called the Chern-Moser curvature tensor. In each case, if the aforementioned tensor vanishes, then the reference point $p \in M$ is called a \emph{CR umbilical point}. The CR hypersurface is spherical near $p$ if and only if the tensor vanishes in a neighborhood of $p$. Furthermore, Graham \cite{Graham1987a} showed that, in the case $n=2$, the obstruction tensor $\mathcal O$ equals $4A^0_{44}$.

\section{Approximation by obstruction flat hypersurfaces}\label{sec:3}

We first recall the definition of osculating a hypersurface in $\bC^n$. Let $M$ be a strongly pseudoconvex hypersurface in $\mathbb{C}^n$ with $p\in \mathbb{C}^n$. Let
$(z, w)=(z_1,\cdots, z_{n-1}, w)\in \mathbb{C}^n$ be a coordinate chart, vanishing at $p$, in which $M$ is given as a graph of the form
\begin{equation}\label{graph}
	2u=|z|^2+\phi(z,\bar z,v)
\end{equation}
where $w=u+iv$ and $\phi$ is $O_{wt}(3)$, i.e., vanishes to weighted order at least 3 at $(z,v)=(0,0)$; here, the weights of $z$ and $w$ are as in Section \ref{sec:2}. Now, let $M'$ be another hypersurface and assume that, in the coordinate system $(z,w)$,  it is given as a graph of the form \eqref{graph} with graphing function $|z|^2+\phi'(z,\bar z,v)$.

\begin{defn}\label{def osculation}
The hypersurface $M'$ osculates $M$ (and vice versa) to weighted order $k$ if $\phi-\phi'=O_{wt}(k)$.
\end{defn}

It is well-known that when $n=2$, $M$ can be osculated by a spherical hypersurface to weighted order 6 at $p$, and to $7$th order if and only if $M$ is CR umbilical at $p$. When $n\geq 3$, $M$ can be osculated by a spherical hypersurface to weighted order 4 at $p$, and to $5$th order if and only if $M$ is CR umbilical at $p$.  We shall show that a similar phenomenon holds with osculation by a spherical hypersurface replaced by osculation by an obstruction flat one. Of particular note is that osculation by an obstruction flat hypersurface can be achieved to a significantly higher weighted order, increasing with the dimension $n$, than osculation by a spherical one. This, of course, reflects the fact that, locally, being obstruction flat is weaker than being spherical.

\begin{thm}\label{thm:obsosc}
	Let $M\subset \mathbb{C}^n$ be a strongly pseudoconvex hypersurface and $p\in M$.
	\begin{itemize}
		\item [(a)] $M$ can be osculated to weight $(2n+4)$th order at $p$ by an obstruction flat hypersurface.
		\item [(b)] $M$ can be osculated to weight $(2n+5)$th order at $p$ by an obstruction flat hypersurface if and only if $M$ is obstruction flat at $p$.
	\end{itemize}
\end{thm}

\begin{proof}
We apply the idea in \cite{Graham1987a} (see Proposition 4.14 there) in the proof.
First pick a local coordinates system $(z, w)\in \mathbb{C}^{n-1}\times \mathbb{C}$ near $p$ such that $p=0$ and $M$ is in the Chern--Moser normal form \eqref{CM normal form} near $p$. We introduce the defining function
	\begin{equation*}
		\rho:=2u-|z|^2-\sum_{\underset{l\geq 0}{|\alpha|, |\beta|\geq 2}} A_{\alpha\bar{\beta}}^l z^{\alpha}\oo{z^{\beta}} v^l,
	\end{equation*}
	where $w=u+iv\in \mathbb{C}$. In the following, we denote  $z=(z_1, \cdots, z_{n-1})=(x_1+iy_1, \cdots, x_{n-1}+iy_{n-1})$. Let $M'\subset \mathbb{C}^n$ be a strongly pseudoconvex hypersurface and $\psi$ a defining function for $M'$. We recall the recursive formula \eqref{Fefferman algorithm} for obtaining a Fefferman defining function $\psi_{n+1}$. 
	We then set $K(\psi):=\bigl(J(\psi_{n+1})-1\bigr)/\psi_{n+1}^{n+1}$. As observed in \cite{Graham1987a}, $K(\psi)$ is a nonlinear differential operator of order $2n+4$, depending real analytically on $\psi$ and its derivatives. Since $\psi_{n+1}$ is a Fefferman defining function for $M'$, by \eqref{Fefferman defining function and obstruction function} $\frac{1}{n+2}K(\psi)|_{\psi=0}$ is the obstruction function for $M'$. Following the proof of Proposition 4.14 in \cite{Graham1987a}, we consider the Cauchy problem for $K(\psi)=0$ with Cauchy data on $x_1=0$ given by
	\begin{equation*}
		\psi_0=2u-|z|^2-\sum_{|\alpha|+|\beta|+2l\leq 2n+4} A_{\alpha\bar{\beta}}^l z^{\alpha}\oo{z^{\beta}} v^l+\frac{1}{(n+2)^2}K(\rho)(0)\,|z_1|^{2n+4}.
	\end{equation*}
More precisely, we require $\psi=\psi_0 \mod x_1^{2n+4}$ near the origin. In order to apply the Cauchy-Kowalevski theorem, we need to check the following two conditions which guarantee that the data is consistent and the problem is non-characteristic:
	\begin{align}\label{consistent and noncharacteristic}
		\mbox{(i) }K(\psi_0)(0)=0 \qquad \mbox{(ii) } \frac{d}{dt}\Big|_{t=0} K(\psi_0+tx_1^{2n+4})(0)\neq 0.
	\end{align}
	
Since $\frac{1}{(n+2)}K(\psi)$ is the obstruction function for the hypersurface $M'$, by Remark 4.13 in \cite{Graham1987a} we have
	\begin{align}\label{obstruction function and A}
	\begin{split}
		K\Bigl(2u-|z|^2-\sum_{|\alpha|+|\beta|+2l\leq 2n+4}A_{\alpha\bar{\beta}}^l z^{\alpha}&\oo{z^{\beta}}v^l \Bigr)(0)=(n+2)^2\sum_{j=0}^{\lfloor \frac{n-2}{2}\rfloor} 2^{-2j} \binom{n}{j} \binom{n+2}{2j}^{-1}\tr^p A_{p\bar{p}}^{2j}\\
		&+\mbox{nonlinear terms in } A_{\alpha\bar{\beta}}^l \mbox{ with } |\alpha|+|\beta|+2l<2n+4.
	\end{split}
	\end{align}
Here $p=n+2-2j$ and $\tr^p A_{p\bar{p}}^{2j}:=\frac{1}{p!^2}(\sum_{k=1}^{n-1}\partial_{z_k}\partial_{\oo{z_k}})^p\sum_{|\alpha|=|\beta|=p}A_{\alpha\bar{\beta}}^{2j}z^{\alpha}\oo{z^{\beta}}=\frac{1}{p!}\sum_{|\alpha|=p} A_{\alpha\bar{\alpha}}^{2j} \,\alpha!$ which is linear in terms $A_{\alpha\bar{\beta}}^l$ with $|\alpha|+|\beta|+2l=2n+4$.

Note $\psi_0$ only differ from $2u-|z|^2-\sum_{|\alpha|+|\beta|+2l\leq 2n+4}A_{\alpha\bar{\beta}}^l z^{\alpha}\oo{z^{\beta}}v^l$ by the term $\frac{1}{(n+2)^2}K(\rho)(0)|z_1|^{2n+4}$, which is about the linear coefficient $A_{\alpha\bar{\beta}}^0$ for $\alpha=\beta=(n+2, 0, \cdots, 0)$ in \eqref{obstruction function and A}. Thus \eqref{obstruction function and A} implies
\begin{equation*}
	K(\psi_0)(0)=K\Bigl(2u-|z|^2-\sum_{|\alpha|+|\beta|+2l\leq 2n+4}A_{\alpha\bar{\beta}}^l z^{\alpha}\oo{z^{\beta}}v^l \Bigr)(0)-K(\rho)(0)=0.
\end{equation*}
The last equality is because the value of the obstruction function at $p=0$ is independent of terms $A_{\alpha\bar{\beta}}^{2l}z^{\alpha}\oo{z^{\beta}}v^l$ with $|\alpha|+|\beta|+2l>2n+4$ (see \cite{Graham1987a}). Therefore, (i) in \eqref{consistent and noncharacteristic} is verified.

Similarly, by \eqref{obstruction function and A} we also have
\begin{align*}
	K\bigl(\psi_0+tx_1^{2n+4}\bigr)(0)=\frac{(n+2)^2}{(n+2)!^2} \Bigl(\sum_{k=1}^{n-1}\partial_{z_k}\partial_{\oo{z_k}}\Bigr)^{n+2} tx_1^{2n+4} \Big|_{z=0}=\frac{(n+2)^2}{2^{2n+4}}\binom{2n+4}{n+2}\, t.
\end{align*}
Thus (ii) in \eqref{consistent and noncharacteristic} follows immediately.

Now we can apply the Cauchy-Kowalevski theorem, and obtain a real analytic function $\chi$ near the origin so that $\psi=\psi_0+\chi\, x_1^{2n+4}$ satisfies $K(\psi)=0$. Therefore, the hypersurface $M'$ defined by $\psi=0$ is obstruction flat. Moreover, by \eqref{obstruction function and A} again we get
\begin{equation*}
	K(\psi_0+\chi x_1^{2n+4})(0)=\frac{(n+2)^2}{2^{2n+4}}\binom{2n+4}{n+2}\chi(0).
\end{equation*}
Thus, $\chi(0)=0$ and
\begin{equation*}
	\rho-\psi=-\frac{1}{(n+2)^2}K(\rho)(0)\, |z_1|^{2n+4}+ \mbox{terms of weights} \geq 2n+5.
\end{equation*}
So  statements (a) and (b) now follow directly from Definition \ref{def osculation}.
\end{proof}

\begin{rmk}
	When $n=2$, we actually have $K(\rho)(0)=16 A_{4\bar{4}}^0$ and \eqref{obstruction function and A} is simplified into
	\begin{equation*}
		K\Bigl(2u-|z|^2-\sum_{\alpha+\beta+2l\leq 8}A_{\alpha\bar{\beta}}^l z^{\alpha}\oo{z^{\beta}}v^l \Bigr)(0)=16 A_{4\bar{4}}^0.
	\end{equation*}
	The conditions in \eqref{consistent and noncharacteristic} follow directly from this formula.
\end{rmk}

\section{Obstruction flatness versus sphericity}\label{sec:4}
Recall that, by Prop 4.14 in \cite{Graham1987a}, the collection of (locally) obstruction flat, strongly pseudoconvex CR hypersurfaces properly contains that of spherical ones. We will show that the proper containment still holds even if we restrict to CR hypersurfaces with transverse symmetry. By \cite{EXX22} (see Corollary 1.3 there and the discussion right after it), for any odd integer $m \geq 5$, there are even {\em compact}, obstruction flat and
non-spherical $m$-dimensional CR hypersurfaces with transverse symmetry. By \cite{Ebenfelt18}, however, in the $3$-dimensional case such compact CR manifolds do not exist. Our next result shows that, locally, there are (many) obstruction flat and
non-spherical $3$-dimensional CR hypersurfaces with transverse symmetry.

\begin{thm}\label{existence of obstruction flat hypersurface with transverse symmetry Prop}
	There are real analytic, strongly pseudoconvex  CR hypersurface $M \subset \mathbb{C}^2$ with transverse symmetry, such that $M$ is obstruction flat but not spherical.
\end{thm}



To prove Theorem \ref{existence of obstruction flat hypersurface with transverse symmetry Prop}, we let $U$ be an open connected subset of $\mathbb{C}$ and consider the trivial line bundle $L=U \times \mathbb{C}$ over $U$. Let $h$ be a Hermitian metric on $L$ such that $(L, h)$ is negative, i.e.,
$\omega:=\sqrt{-1}\partial\dbar\log h$ is a \k form on $U$. By setting  $\phi(z,\bar{z}):=\frac{1}{2}\log\Bigl(\partial_z\partial_{\bar{z}}\log h(z,\bar{z})\Bigr)$, the \k form $\omega$ and \k metric can be respectively expressed as
\begin{equation}\label{Kahler metric}
	\omega=\sqrt{-1}\, e^{2\phi} dz\wedge d\bar{z}, \qquad g=e^{2\phi}(dz\otimes d\bar{z}+d\bar{z}\otimes dz)=2e^{2\phi} |dz|^2.
\end{equation}
The Gauss curvature is given by
\begin{equation}\label{Gauss curvature}
	K=-2e^{-2\phi}\partial_{z}\partial_{\bar{z}}\phi.
\end{equation}
Note that $\phi$ and $K$ are both real valued. The circle bundle of $L$, which we call $M$, is defined as
\begin{equation}\label{circle bundle}
	|\xi|^2 h(z,\bar{z})=1 \quad \mbox{ for any } (z, \xi)\in U\times \mathbb{C}.
\end{equation}
It is clear that $M$ has a $U(1)$-transverse symmetry (with $U(1)$ acting on the $\xi$-variable).
In the case of a circle bundle over Riemann surface, sphericity and obstruction flatness of $M$ can both be characterized in terms of the Gauss curvature. This follows from general formulae in \cite{ChengLee1990,Hirachi2014} and the work of Webster \cite{Web77b}; it was also proved directly in this context in \cite{Ebenfelt18} (cf. \cite{ES17,Wang}).
\begin{prop} \label{sphericity and obstruction flatness characterization Prop}
	The sphericity and obstruction flatness of $M$ are characterized respectively as follows.
	\begin{itemize}
		\item $M$ is spherical if and only if $K_{;\bar{z}\bar{z}}=0$.
		\item Up to a nowhere vanishing, real-valued function, the obstruction function $\mathcal{O}$ is $K_{;\bar{z}\bar{z}zz}$. In particular, $M$ is obstruction flat if and only if $K_{;\bar{z}\bar{z}zz}=0$.
	\end{itemize}
\end{prop}
Here $K_{;\bar{z}\bar{z}}$ and $K_{;\bar{z}\bar{z}zz}$ stand for the repeated covariant derivatives with respect to $\frac{\partial}{\partial z}$ and $\frac{\partial}{\partial\bar{z}}$: Let $\nab$ be the Levi-Civita connection for $(X, g)$, by denoting $\nab_z:=\nab_{\frac{\partial}{\partial z}}$ and $\nab_{\bar{z}}:=\nab_{\frac{\partial}{\partial \bar{z}}}$, we then have $K_{;\bar{z}\bar{z}}=\nab_{\bar{z}}\nab_{\bar{z}}K$ and $K_{;\bar{z}\bar{z}zz}=\nab_z\nab_z\nab_{\bar{z}}\nab_{\bar{z}}K$.

\begin{lemma}\label{lemma sphericity and obstruction flat PDE}
	We have
	\begin{equation}\label{sphericity and obstruction flat PDE}
	 e^{-4\phi}K_{;\bar{z}\bar{z}zz}=\frac{1}{4}\Delta_g^2K+\frac{1}{4}\Delta_g K^2,
	\end{equation}
	where $\Delta_g=2e^{-2\phi}\nab_{z}\nab_{\bar{z}}$ is the (real) Laplacian on $(X, g)$.
	In particular, $K_{;\bar{z}\bar{z}zz}$ is real valued.
\end{lemma}

\begin{proof}
	By the Ricci identity, we have
	\begin{align*}
		K_{;zz\bar{z}}-K_{;z\bar{z}z}=-K_{;z} R_{1\bar{1}1}{}^{1}=K_{;z}R_{1\bar{1}}=K_{;z}K e^{2\phi}=\frac{1}{2} e^{2\phi} (K^2)_{;z} ,
	\end{align*}
	where $R_{1\bar{1}1\bar{1}}=-g\bigl(\nab_z\nab_{\bar{z}}\frac{\partial}{\partial z}-\nab_{\bar{z}}\nab_z\frac{\partial}{\partial z}, \frac{\partial}{\partial \bar{z}}\bigr)$ is the Riemannian curvature and the index can be raised up by the metric; and $R_{1\bar{1}}=-R_{1\bar{1}1}{}^1$ is the Ricci curvature.
	
	By taking one more covariant derivative, we get
	\begin{align*}
		K_{;zz\bar{z}\bar{z}}=K_{;z\bar{z}z\bar{z}}+\frac{1}{2}e^{2\phi} (K^2)_{;z\bar{z}}=\frac{1}{4}e^{4\phi}\Delta_g^2K+\frac{1}{4} e^{4\phi} \Delta_g K^2.
	\end{align*}
	So the result follows by taking the conjugate.
\end{proof}

The following lemma on the regularity of $h$ and $\phi$ for obstruction flat hypersurfaces in $\mathbb{C}^2$  will not be directly used in the proof of Proposition \ref{existence of obstruction flat hypersurface with transverse symmetry Prop}. We record it here for completeness.

\begin{prop}\label{lem phi is real analytic}
	Suppose $h$ is $C^8$ (i.e., $\phi$ is $C^6$) on $U$. If $M$ is obstruction flat, then $h$ and $\phi$ are real analytic on $U$.
\end{prop}

\begin{proof}
By Proposition \ref{sphericity and obstruction flatness characterization Prop} and Lemma \ref{lemma sphericity and obstruction flat PDE}, $M$ is obstruction flat if and only if
	\begin{equation*}
		\Delta_g^2K+\Delta_g K^2=0.
	\end{equation*}
	In terms of the local coordinates, we can write it into
	\begin{equation*}
		4e^{-2\phi} \partial_z\partial_{\bar{z}} (e^{-2\phi}\partial_z\partial_{\bar{z}} K)+ 2e^{-2\phi}\partial_z\partial_{\bar{z}} K^2=0.
	\end{equation*}
	By \eqref{Gauss curvature}, we can further rewrite it into a PDE on $\phi$:
	\begin{equation}\label{PDE in local coordinates}
		\Delta \bigl(e^{-2\phi}\Delta (e^{-2\phi}\Delta \phi)\bigr)-\Delta \bigl(e^{-4\phi} (\Delta\phi)^2\bigr) =0,
	\end{equation}
	where $\Delta=4\partial_z\partial_{\bar{z}}$ is the Laplacian on Euclidean space. Note that the left hand side of \eqref{PDE in local coordinates} is a partial differential operator of order $6$, depending real analytically on $\phi$ and its derivatives. Moreover, the highest order term on $\phi$ is $\Delta^3\phi$ up to a positive factor $e^{-4\phi}$, and thus \eqref{PDE in local coordinates} is elliptic. By the analytic hypoellipticity, $\phi$ is real analytic. Recall that $\partial_z\partial_{\bar{z}}\log h=e^{2\phi}$. By the analytic hypoellipticity of $\Delta$, $h$ is also real analytic. So the proof is completed.
\end{proof}


Let $\mathcal{C}_0^{\omega}(\mathbb{R})$ be the set of germs of real valued  $C^{\omega}$ (real analytic) functions at the origin in $\mathbb{R}$ and $\mathcal{C}_0^{\omega}(\mathbb{C})$ be the set of germs of real valued  $C^{\omega}$ functions at the origin in $\mathbb{C}$. Given any $\phi=\phi(z,\bar{z})\in \mathcal{C}_0^{\omega}(\mathbb{C})$, it is $C^{\omega}$ on some open neighborhood $V\subset \mathbb{C}$ containing the origin, and we define a \k metric $\omega$ on $V$ by \eqref{Kahler metric}. Then we solve $\partial_z\partial_{\bar{z}}\log h(z,\bar{z})=e^{\phi(z,\bar{z})}$ to find a positive function $h$ on $V$.   We then define a germ of strongly pseudoconvex hypersurface $M_{\phi, h}$ as in the equation of \eqref{circle bundle} for $(z, \xi) \in V \times \mathbb{C}$. Note that while the choice of $h$ is not unique, any other choice $\widetilde{h}$ must be such that $\log h$ and $\log \widetilde{h}$ differ by a harmonic function on $V$. Consequently, $h=\widetilde{h} |e^f|^2$ for some holomorphic function $f$ near the origin in $\mathbb{C}$. This implies that, by shrinking $V$ if needed, the hypersurfaces $M_{\phi, h}$ and $M_{\phi, \widetilde{h}}$ are CR diffeomorphic via the map $(z, \xi)\rightarrow (z, e^f\xi)$. For this reason, we will simply denote $M_{\phi, h}$ by $M_{\phi}$.
Set

$$\Gamma_s:=\{\phi\in \mathcal{C}_0^{\omega}(\mathbb{C}): \mbox{$M_{\phi}$ is spherical} \}; \quad
	\Gamma_{ob}:=\{ \phi\in \mathcal{C}_0^{\omega}(\mathbb{C}): \mbox{$M_{\phi}$ is obstruction flat} \}.$$

In next lemma, we will give a parametrization of $\Gamma_{ob}$. 
\begin{lemma}\label{lem parameterization of obstruction flat hypersurfaces}
	Let $\Phi_{ob}: \mathcal{C}^{\omega}_0(\mathbb{C}) \rightarrow (\mathcal{C}^{\omega}_0(\mathbb{R}))^6$ be the map defined by
	\begin{equation*}
		\Phi_{ob}(\phi)=\Bigl(\phi, \frac{\partial \phi}{\partial y}, \frac{\partial^2\phi}{\partial y^2}, \cdots, \frac{\partial^5 \phi}{\partial y^5}\Bigr)\Big|_{y=0},
	\end{equation*}
	where $\phi=\phi(z,\bar{z})$ and $z=x+iy$.
	Then the restriction $\Phi_{ob}: \Gamma_{ob}\rightarrow (\mathcal{C}^{\omega}_0(\mathbb{R}))^6$ is bijective.
\end{lemma}

\begin{proof}
Recall that by Proposition \ref{sphericity and obstruction flatness characterization Prop}, $M_{\phi}$ is obstruction flat if and only if $K_{;\bar{z}\bar{z}zz}=0$, which by Lemma \ref{lemma sphericity and obstruction flat PDE} and equation \eqref{PDE in local coordinates} is equivalent to the PDE:
	\begin{equation*}
		e^{4\phi}\Delta \bigl(e^{-2\phi}\Delta (e^{-2\phi}\Delta \phi)\bigr)-e^{4\phi}\Delta \bigl(e^{-4\phi} (\Delta\phi)^2\bigr) =0.
	\end{equation*}
	Note that the highest order term on $\phi$ is $\Delta^3\phi$ and we can write the above equation into the form:
	\begin{equation}\label{Cauchy-Kowalevski obstruction flat}
		\frac{\partial^6\phi}{\partial y^6}=F\Bigl(x, y, \frac{\partial^{i+j}\phi}{\partial x^i \partial y^j}\Big|_{i+j\leq 6, j<6}\Bigr)
	\end{equation}
	for some real analytic function $F$. By the Cauchy-Kowalevski theorem, there exists a real analytic solution $\phi$ for \eqref{Cauchy-Kowalevski obstruction flat} subject to the Cauchy data given by an element in $(\mathcal{C}_0^{\omega}(\mathbb{R}))^6$ along $y=0$. Thus $\Phi_{ob}$ is surjective. Since the Cauchy-Kowalevski theorem also guarantees the uniqueness of the solution in the category of real analytic functions, the map $\Phi_{ob}$ is also injective. So the proof is completed.
\end{proof}

Similarly, we also give a description of $\Gamma_s$.
\begin{lemma}\label{lem parameterization of spherical hypersurfaces}
	Let $\Phi_{s}: \mathcal{C}^{\omega}_0(\mathbb{C}) \rightarrow (\mathcal{C}^{\omega}_0(\mathbb{R}))^4$ be the map defined by
	\begin{equation*}
		\Phi_{s}(\phi)=\Bigl(\phi, \frac{\partial \phi}{\partial y}, \frac{\partial^2\phi}{\partial y^2}, \frac{\partial^3 \phi}{\partial y^3}\Bigr)\Big|_{y=0},
	\end{equation*}
	where $\phi=\phi(z,\bar{z})$ and $z=x+iy$.
	Then the restriction $\Phi_{s}: \Gamma_{s}\rightarrow (\mathcal{C}^{\omega}_0(\mathbb{R}))^4$ is injective.
\end{lemma}
\begin{proof}
Recall that by Proposition \ref{sphericity and obstruction flatness characterization Prop}, $M_{\phi}$ is spherical if and only if $K_{;\bar{z}\bar{z}}=0$. By a straightforward computation, we can express $K_{;\bar{z}\bar{z}}$ in local coordinates:
	\begin{equation*}
		K_{;\bar{z}\bar{z}}=\partial^2_{\bar{z}}K-2\,\partial_{\bar{z}}\phi \cdot \partial_{\bar{z}}K.
	\end{equation*}
	We further write it in terms of real variables $(x, y)$:
	\begin{align*}
		K_{;\bar{z}\bar{z}}=&\frac{1}{4}(\partial_x+i\partial_y)^2K-\frac{1}{2}\,(\partial_x+i\partial_y)\phi \cdot (\partial_x+i\partial_y)K\\
		=&\frac{1}{4}\partial_x^2K-\frac{1}{4}\partial_y^2K-\frac{1}{2}\partial_x\phi\cdot \partial_xK+\frac{1}{2}\partial_y\phi\cdot\partial_yK+\frac{i}{2}\bigl(\partial_x\partial_yK-\partial_x\phi\cdot\partial_yK-\partial_y\phi\cdot\partial_xK\bigr).
	\end{align*}
	Thus, $K_{;\bar{z}\bar{z}}=0$ if and only if
	\begin{equation}\label{PDE in local coordinates spherical}
		\begin{dcases}
			\partial_x^2K-\partial_y^2K-2\partial_x\phi\cdot \partial_xK+2\partial_y\phi\cdot\partial_yK=0\\
			\partial_x\partial_yK-\partial_x\phi\cdot\partial_yK-\partial_y\phi\cdot\partial_xK=0,
		\end{dcases}
	\end{equation}
where $K$ is given by \eqref{Gauss curvature}. The first equation in \eqref{PDE in local coordinates spherical} can be written into the form
	\begin{equation*}
		\frac{\partial^4\phi}{\partial y^4}=F\Bigl(x, y, \frac{\partial^{i+j}\phi}{\partial x^i \partial y^j}\Big|_{i+j\leq 4, j<4} \Bigr)
	\end{equation*}
	for some real analytic function $F$. Similarly as in the proof of Lemma \ref{lem parameterization of obstruction flat hypersurfaces}, by the Cauchy-Kowalevski theorem $\Phi_s$ is injective.
\end{proof}

We are now ready to prove Theorem \ref{existence of obstruction flat hypersurface with transverse symmetry Prop}.
\begin{proof}[Proof of Theorem \ref{existence of obstruction flat hypersurface with transverse symmetry Prop}]
	Given an element $v\in (\mathcal{C}^{\omega}_0(\mathbb{R}))^4$, by the Cauchy-Kowalevski theorem, there exists some real analytic function $\phi_0$ near the origin of $\mathbb{C}$ satisfying the first equation in \eqref{PDE in local coordinates spherical} and $\Phi_s(\phi_0)=v$. For this solution $\phi_0$, $\Phi_{ob}(\phi_0)+(0, 0, 0, 0, 0, 1)$ gives an element in $(\mathcal{C}^{\omega}_0(\mathbb{R}))^6$, by Lemma \ref{lem parameterization of obstruction flat hypersurfaces} there exists some real analytic function $\phi$ near the origin of $\mathbb{C}$ satisfying \eqref{PDE in local coordinates} and
	\begin{equation}\label{modify the jet data}
		\Phi_{ob}(\phi)=\Phi_{ob}(\phi_0)+(0, 0, 0, 0, 0, 1).
	\end{equation}
	
	Since $\phi$ satisfies \eqref{PDE in local coordinates}, the hypersurface $M_{\phi}$ in $\mathbb{C}^2$ is obstruction flat. As $M_{\phi}$ is a circle bundle, it is of transverse symmetry. Lastly, we check that $M_{\phi}$ is not spherical. Suppose $M_{\phi}$ is spherical. Then both $\phi$ and $\phi_0$ satisfy the first equation in \eqref{PDE in local coordinates spherical}. We also note that $\frac{\partial^j\phi}{\partial y}|_{y=0}=\frac{\partial^j\phi_0}{\partial y}|_{y=0}$ for $0\leq j\leq 3$ by \eqref{modify the jet data}. Thus, the Cauchy-Kowalevski theorem implies that $\phi=\phi_0$ as they are both real analytic. However, this contradicts \eqref{modify the jet data}. So the proof is completed.
\end{proof}

\section{The existence of obstruction flat points}\label{sec:5}
The  following questions originate in the work of S.-S. Chern and J. K. Moser \cite{ChMo74}. Are there compact strongly pseudoconvex 3-dimensional CR hypersurfaces without CR umbilical points? Can such a CR hypersurface exist in $\mathbb{C}^2$? The answers to both questions turn out to be affirmative: An example in $\mathbb{CP}^2$ was actually found earlier by E. Cartan in \cite{Car33} but it is not embeddable into $\mathbb{C}^2$. More recently, D. Son, D. Zaitsev and the first-named author \cite{ESZ18} constructed an example in $\mathbb{C}^2$. On the other hand, when $M$ is the circle bundle of a negative line bundle over a compact Riemann surface $X$, D. Son and the first-named author \cite{ES17} proved that $M$ has at least a circle of CR umbilical points provided that $X$ is not a torus. Here, we prove an analogous result for  obstruction flat points.

\begin{thm}\label{existence of obstruction flat points on circle bundle Prop}
	Let $(L, h)$ be a negative line bundle over a compact Riemann surface $X$, so that the dual bundle $(L^*, h^{-1})$ of $(L, h)$ induces a \k metric $g$ on $X$. Let $M$ denote the unit circle bundle in $(L, h)$. Then
\begin{itemize}
\item[$(1)$] $M$  has at least a circle of obstruction flat points. 
\item[$(2)$] Let $\mathcal{O}$ be the obstruction function of $M$. If $\mathcal{O}\geq 0$ on $M$ or $\mathcal{O}\leq 0$ on $M$, then $\mathcal{O}\equiv 0$ on $M$.  Consequently, $M$ is spherical.
\end{itemize}
\end{thm}


\begin{proof}
We first prove part (2). By Proposition \ref{sphericity and obstruction flatness characterization Prop} and Lemma \ref{lemma sphericity and obstruction flat PDE}, the obstruction function $\mathcal{O}$ is $\Delta_g^2K+\Delta_g K^2$ up to some nowhere vanishing function. By the assumption in (2), we have either $\Delta_g^2K+\Delta_g K^2\geq 0$ or $\Delta_g^2K+\Delta_g K^2\leq 0$ on $X$. On the other hand, as $\Delta_g^2K+\Delta_g K^2$ is in the image of the Laplace operator $\Delta_g$, we have $\int_X \Delta_g^2K+\Delta_g K^2 dV_g=0$. Therefore, $\Delta_g^2K+\Delta_g K^2\equiv 0$ on $X$ and thus $\mathcal{O}\equiv 0$ on $M$. The last assertion follows from \cite{Ebenfelt18}.
	
To prove part (1),  we note that by part (2), if $\mathcal{O} \not \equiv 0,$ then there are points where $\mathcal{O}$ is positive and points where it is negative; so is $\Delta_gK+\Delta_gK^2$ on $X$. The existence of obstruction flat point then follows by the continuity of the latter function. As the $S^1$ action on $M$ is CR diffeomorphic, we obtain an orbit of obstruction flat points. So the proof is completed.
\end{proof}

In the higher dimensional case, there exist compact CR hypersurfaces with no CR umbilical points. Webster \cite{Web00} proved that a real ellipsoid without circular sections in $\mathbb{C}^{n+1}$ with $n\geq 2$ has no CR umbilical points. Note that such real ellipsoids do not admit a transverse symmetry. However, the following result shows there also exist compact unit circle bundles $M$ (which in particular have transverse symmetry) of CR dimension $n\geq 2$ with no CR umbilical points.

\begin{prop}\label{circle bundle with no CR umbilical points Prop}
	Let $(L, h)$ be a negative line bundle over a compact complex manifold $X$ of complex dimension $n\geq 2$, so that the dual bundle $(L^*, h^{-1})$ of $(L, h)$ induces a \k metric $g$ on $X$. Assume the universal covering $\widetilde{X}$ of $X$ with the pullback metric $\widetilde{g}$ of $g$ is a homogeneous \k manifold. Then the circle bundle $M$ of $(L, h)$ has no CR umbilical points provided $(\widetilde{X}, \widetilde{g})$ is not any of the following:
	\begin{itemize}
		\item [(1)] $(\mathbb{B}^n, \lambda \,\omega_{-1})$ for some $\lambda\in \mathbb{R}^+$,
		\item [(2)] $(\mathbb{CP}^n, \lambda\, \omega_1)$ for some $\lambda\in \mathbb{R}^+$,
		\item [(3)] $(\mathbb{C}^n, \omega_0)$,
		\item [(4)] $(\mathbb{B}^l\times \mathbb{CP}^{n-l}, \lambda \omega_{-1}\times \lambda\omega_1)$ for some $1\leq l\leq n-1$ and some $\lambda\in\mathbb{R}^+$.
	\end{itemize}
	Here $\omega_c$ denotes the \k metric with constant holomorphic sectional curvature $c$.
\end{prop}

\begin{proof}
	Sine $(\widetilde{X}, \widetilde{g})$ is a homogeneous \k manifod, $(X, g)$ is a locally homogeneous \k manifold, and consequently $M$ is a locally homogeneous CR manifold. Therefore, $M$ is either spherical (i.e., CR umbilical everywhere) or has no CR umbilical point at all. Following the work of \cite{Web78}, \cite{Bry01} and \cite{Wang} (cf. Proposition 1.12 in \cite{EXX22}), the sphericity of $M$ implies $(\widetilde{X}, \widetilde{g})$ is one in the list (1)-(4). So the proof is completed.
\end{proof}

\begin{rmk}
	Note that  \k manifolds $(X, g)$ satisfying the conditions in Proposition \ref{circle bundle with no CR umbilical points Prop} are abundant. For examples, $(X, g)$ can be any compact homogeneous Hodge manifold (i.e., simply connected homogeneous \k manifold) other than $\mathbb{CP}^n$. In particular, $(X, g)$ can be any compact Hermitian symmetric spaces (e.g. $\mathbb{CP}^{n_1}\times \mathbb{CP}^{n_2}$ or any Grassmannian manifolds) other than $\mathbb{CP}^n$.
\end{rmk}

\begin{rmk}
	In Proposition \ref{circle bundle with no CR umbilical points Prop}, as $(\widetilde{X}, \widetilde{g})$ is assumed to be homogeneous, $(X, g)$ is locally homogeneous. It follows that $(X, g)$ has constant Ricci eigenvalues. By Theorem 1.1 in \cite{EXX22}, the circle bundles constructed in Proposition \ref{circle bundle with no CR umbilical points Prop} are actually always obstruction flat while having no CR umbilical points.
\end{rmk}

Motivated by the question of existence of CR umbilical points mentioned at the beginning of this section, as well as Theorem \ref{existence of obstruction flat points on circle bundle Prop} and the osculation result Theorem \ref{thm:obsosc}, one may ask: Do compact strongly pseudoconvex CR hypersurfaces always have obstruction flat points? We suspect that, unlike the situation with CR umbilical points, the answer would be affirmative. Theorem \ref{existence of obstruction flat points on circle bundle Prop} provides an affirmative answer in a special case (and in this case the CR manifold has transverse symmetry). We note that there is a significant difference between the question of existence of CR umbilical points versus that of obstruction flat points in that the former involves the vanishing of a curvature tensor, a problem that becomes increasingly "overdetermined" (for $\dim M\geq 7$, cf. \cite{EZ19}) as the dimension increases, while the latter is simply the vanishing of a single real-valued function. Nevertheless, the existence of obstruction flat points is interesting from the point of view of osculation by "model" (obstruction flat) hypersurfaces.



\bibliographystyle{amsalpha.bst}
\bibliography{references}

\end{document}